\documentclass[10pt, notitlepage]{article}

\usepackage{amsmath,amssymb,amsthm,amsfonts}
\usepackage{graphicx}

\newtheorem{theorem}{Theorem}[section]
\newtheorem{corollary}[theorem]{Corollary}
\newtheorem{lemma}[theorem]{Lemma}
\newtheorem{remark}[theorem]{Remark}

\title{Bounding the Porous Exponential Domination Number of Apollonian Networks \footnote{Research partially funded by CURM, the Center for Undergraduate Research, and NSF grant DMS-1148695}}

\author{Joshua Beverly \qquad Mariah Farley \qquad Christopher McClain \qquad Felicia Stover \\Concord University}

\begin{document}

\flushbottom
\maketitle

\begin{center}
\end{center}

\begin{abstract}
Given a graph $G$ with vertex set $V$, a subset $S$ of $V$ is a dominating set if every vertex in $V$ is either in $S$ or adjacent to some vertex in $S$. The size of a smallest dominating set is called the domination number of $G$. We study a variant of domination called porous exponential domination in which each vertex $v$ of $V$ is assigned a weight by each vertex $s$ of $S$ that decreases exponentially as the distance between $v$ and $s$ increases. $S$ is a porous exponential dominating set for $G$ if all vertices in $S$ distribute to vertices in $G$ a total weight of at least 1. The porous exponential domination number of $G$ is the size of a smallest porous exponential dominating set. In this paper we compute bounds for the porous exponential domination number of special graphs known as Apollonian networks.
\end{abstract}

\section{Introduction}

Exponential domination was first introduced in \cite{Dankelmann} and further studied in \cite{Anderson}. Apollonian networks and their applications were independently introduced in \cite{Andrade} and \cite{Doye}, and further studied in \cite{Zhang2006} and \cite{Zhang2014}. We refer the reader to \cite{HaynesDom} and \cite{HaynesFund} for a comprehensive treatment of the topic of domination in graphs and its many variants. General graph theoretic notation and terminology may be found in \cite{West}. Given a graph $G$, we denote its set of vertices by $V(G)$ and its set of edges by $E(G)$. The degree of a vertex $v$ in $G$ is denoted by $d_G(v)$. The distance in $G$ between vertices $x$ and $y$, denoted by $d_G(x,y)$, is defined to be the length of a shortest path in $G$ that joins $x$ and $y$, if such a path exists, and infinity otherwise. The diameter of $G$, denoted $diam(G)$, is the largest such distance: $diam(G)=\max \{d_G(x,y) \mid x,y \in V(G)\}$.

Let $G$ be a graph, $S \subseteq V(G)$, and $v \in V(G)$. The \textit{porous exponential domination weight} of $S$ at $v$ is \[w^*_S(v)=\sum_{u \in S} \frac{1}{2^{d(u,v)-1}}\] and $S$ is a \textit{porous exponential dominating set} for $G$ if $w^*_S(v)\geq 1$ for all $v \in V(G)$. The size of a smallest porous exponential dominating set for $G$ is the \textit{porous exponential domination number} of $G$. and is denoted by 
$\gamma_e^*(G)$. These definitions were first introduced in \cite{Dankelmann}, although that paper is primarily concerned with another variant, $\gamma_e(G)$, called the \textit{nonporous exponential domination number} of $G$. The key difference between porous exponential domination and nonporous exponential domination is whether the distribution of weights from $S$ may ``pass through'' other vertices in $S$, as is evidenced by the slightly different definition of nonporous weight:
\[w_S(v) = 
     \begin{cases}
          \displaystyle \sum_{u \in S} \frac{1}{2^{f(u,v)-1}} & \mbox{if } v \not\in S\\
          2 & \mbox{if } v \in S
     \end{cases}
\]
where $f(u,v)$ is defined to be the length of a shortest path joining $u$ and $v$ in the subgraph induced by $V(G) \setminus (S \setminus \{u\})$ if such a path exists, and infinity otherwise. It is clear that $\gamma_e^*(G) \leq \gamma_e(G)$.

Having defined porous exponential domination, we now define Apollonian networks. Let $G_1$ be a complete graph on three vertices and let $U_1$ = $V(G_1)$. Let $G_2$ be a complete graph on four vertices such that $U_1$ $\subseteq$ $V(G_2)$, and let $U_2 = V(G_2) \backslash V(G_1)$. For $k > 2$ we define $G_k$ and $U_k$ recursively by extending $G_{k-1}$ and $U_{k-1}$ as follows: for each $u \in U_{k-1}$, and for each adjacent pair $\{x,y\} $ of neighbors of $u$ in $G_{k-1}$, we create a new vertex $v \in U_k$ that is adjacent to each of $u,x,y$ in $G_k$. (Consequently, $u$, $v$, $x$, and $y$ are all pairwise adjacent in $G_k$.) We call $G_k$ the \textit{kth Apollonian network}, and for $1 \leq j \leq k$, we call $U_j$ the \textit{jth generation} of vertices in $G_k$. Note that $V(G_k) = \bigcup_{j=1}^{k} U_j$ and $U_k = V(G_k) \setminus V(G_{k-1})$. This recursive process is more easily visualized by starting with a particular planar embedding of $G_1$ and obtaining $G_k$ from $G_{k-1}$ by adding a new vertex to each interior face and triangulating, as shown in Figures \ref{fig:G1} through \ref{fig:G4}. We note, however, that our formal definition above does not depend upon the planar embedding.

\begin{figure}[!htb]
\minipage{0.5\textwidth}
  \includegraphics[width=\linewidth]{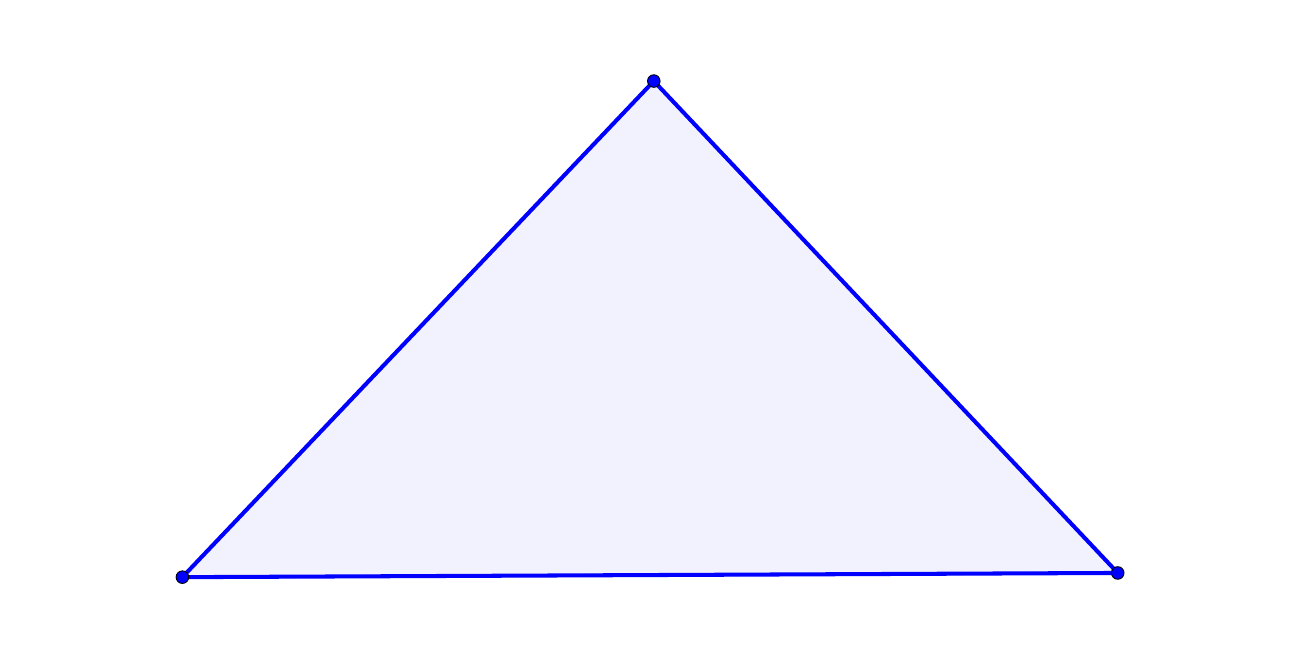}
  \caption{$G_1$}\label{fig:G1}
\endminipage\hfill
\minipage{0.5\textwidth}
  \includegraphics[width=\linewidth]{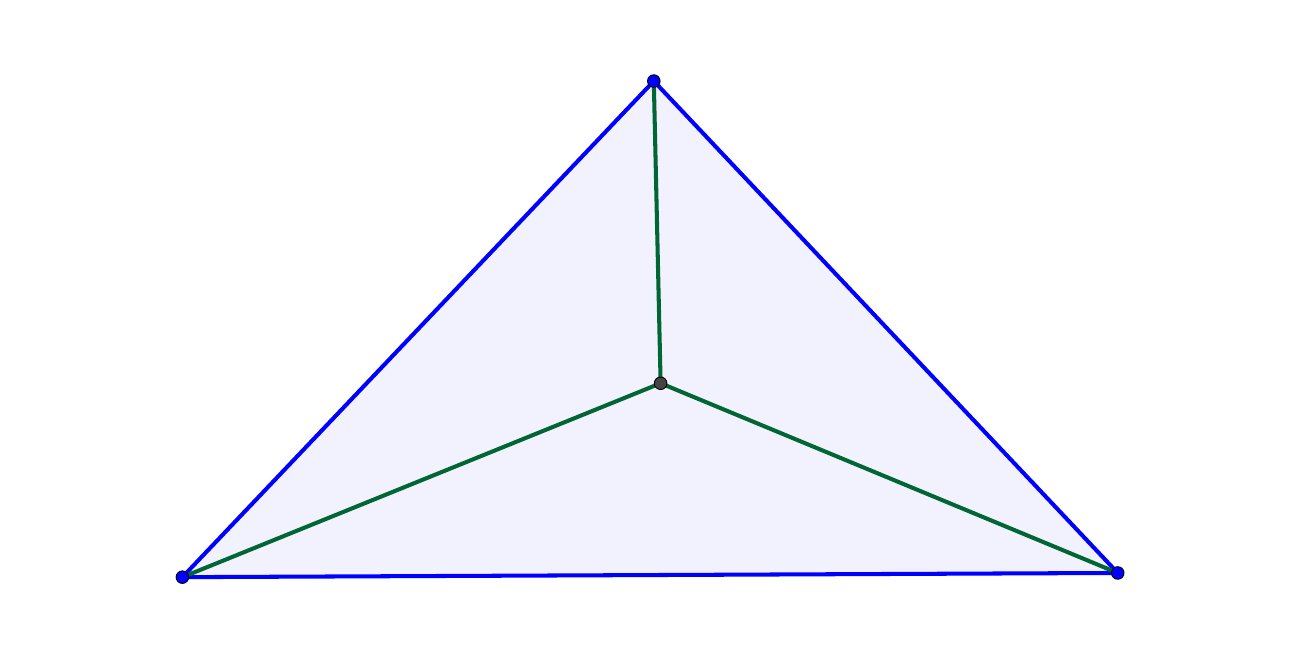}
  \caption{$G_2$}\label{fig:G2}
\endminipage\hfill
\end{figure}

\begin{figure}[!htb]
\minipage{0.5\textwidth}
  \includegraphics[width=\linewidth]{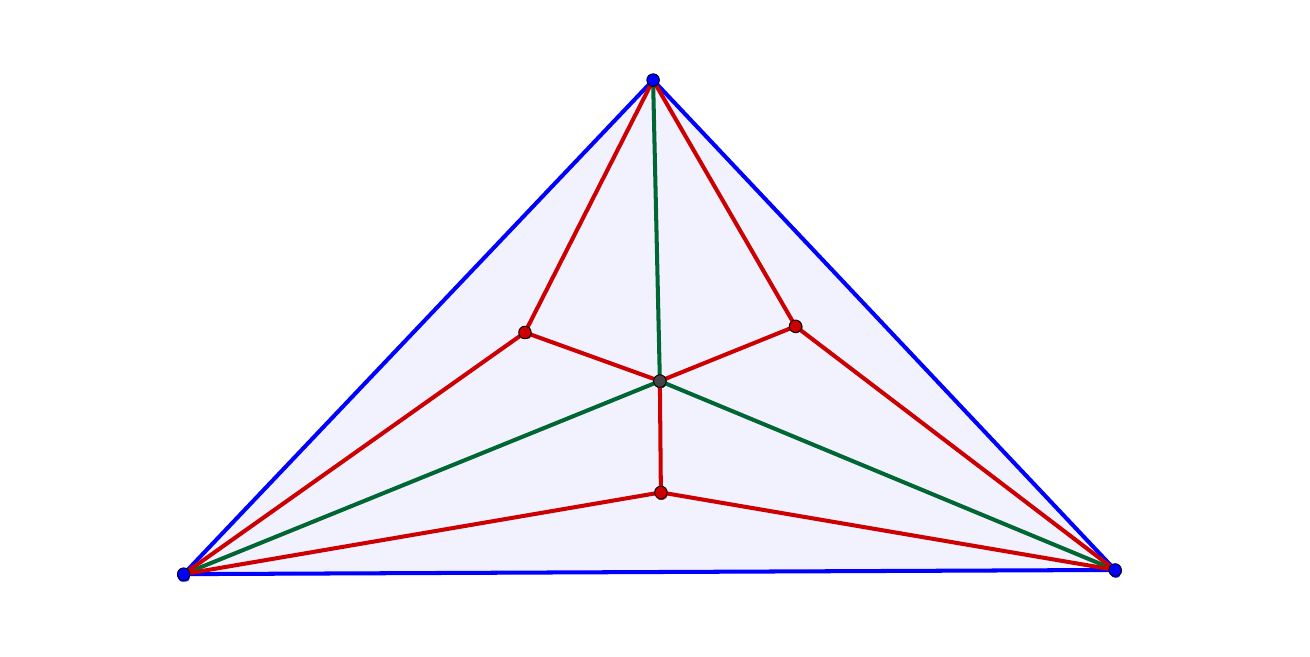}
  \caption{$G_3$}\label{fig:G3}
\endminipage\hfill
\minipage{0.5\textwidth}
  \includegraphics[width=\linewidth]{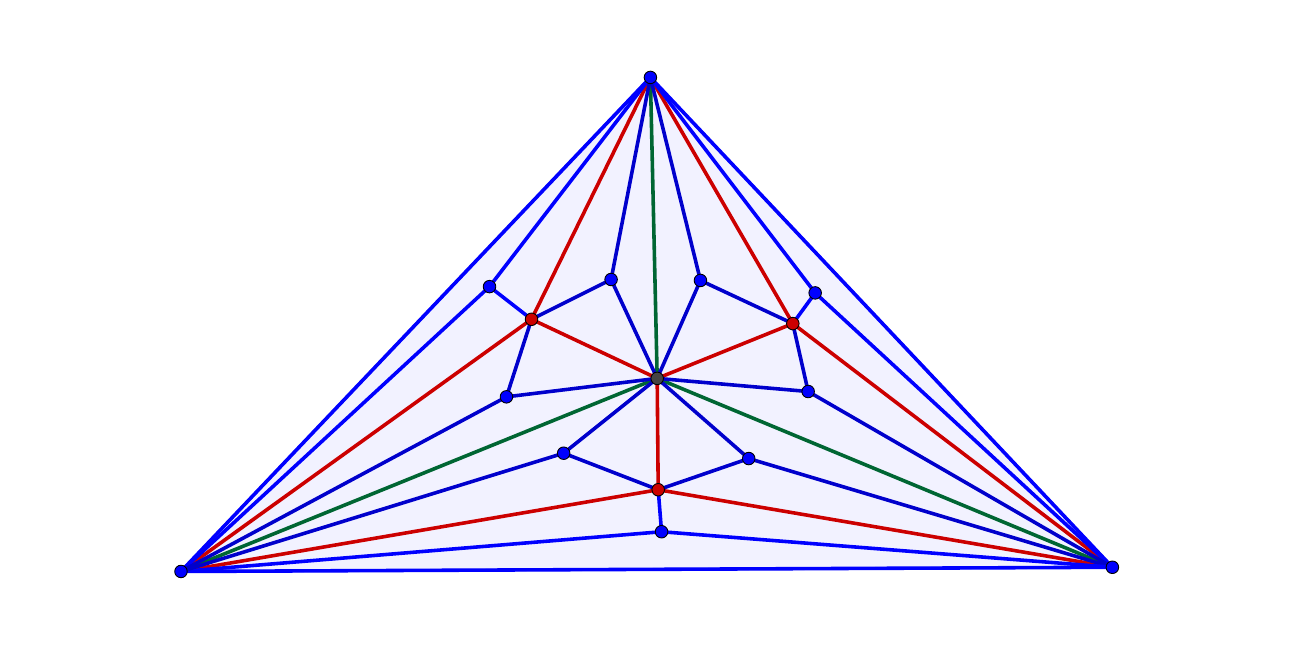}
  \caption{$G_4$}\label{fig:G4}
\endminipage\hfill
\end{figure}

Before stating our main results, we record a few elementary facts based upon our construction of $G_k$ and observation of small cases:

\begin{remark}
\label{Vertices}
$|U_1|=3$, $|U_k|=3^{k-2}$ for $k>1$, and $ \displaystyle |V(G_k)|=\sum_{j=1}^k U_j=3+\sum_{j=0}^{k-2} 3^j=\frac{3^{k-1}+5}{2}$.
\end{remark}

\begin{remark}
\label{Edges}
$\displaystyle |E(G_k)|=3+\sum_{j=2}^{k}3|U_j|=3+\sum_{j=2}^{k} 3^{j-1}=\frac{3^k+3}{2}$.
\end{remark}

\begin{remark}
\label{gammaG3}
Since every vertex in $V(G_3)$ is adjacent to the single vertex in $U_2$, we know that $\gamma_e^*(G_3)=1$.
\end{remark}

\begin{remark}
\label{gammaG5}
Let $S$ be any pair of vertices from $V(G_2)$. Since every vertex in $V(G_5)$ is adjacent to at least one of the vertices in $V(G_2)$ and every pair of vertices in $V(G_2)$ is adjacent, we know that every vertex of $V(G_5)$ is within distance 2 of both vertices in $S$ and therefore $\gamma_e^*(G_5)=2$. (See Figure \ref{fig:G5}.)
\end{remark}

\begin{figure}
  \begin{center}
  \includegraphics[scale=1.2]{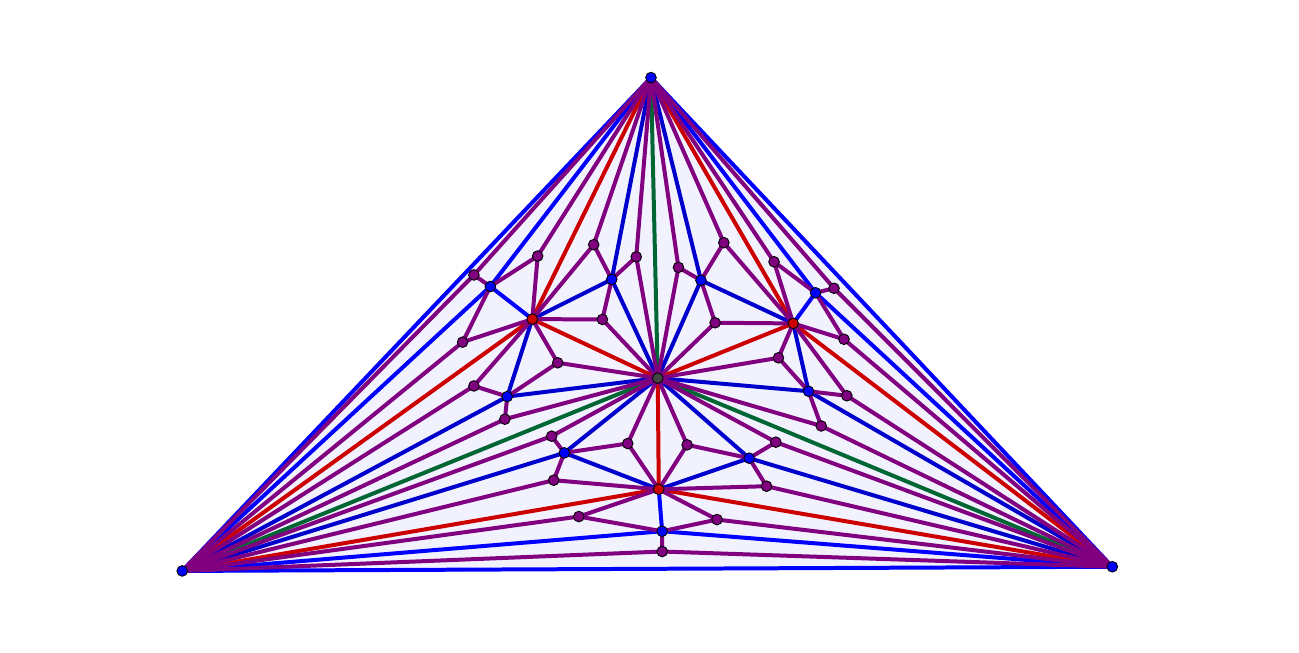}
  \caption{$G_5$}\label{fig:G5}
  \end{center}
\end{figure}

We further invite the reader to verify our observations and computations for the order, diameter, and porous exponential domination number of $G_k$ for $k \leq 7$, as presented in Table \ref{GkSmall} below.

\begin{table}
\begin{center}
\begin{tabular}{|c|c|c|c|c|}

\hline
$k$ & $|V(G_k)|$ & $|E(G_k)|$ & $diam(G_k)$ & $\gamma_e^*(G_k)$ \\
\hline
$1$ & 3 & 3 & 1 & 1 \\
$2$ & 4 & 6 & 1 & 1 \\
$3$ & 7 & 15 & 2 & 1 \\
$4$ & 16 & 42 & 3 & 2 \\
$5$ & 43 & 123 & 3 & 2 \\
$6$ & 124 & 366 & 4 & 3 \\
$7$ & 367 & 1095 & 5 & 3 \\
\hline

\end{tabular}
\end{center}
\caption{Observations for $G_k$, $k \leq 7$}
\label{GkSmall}
\end{table}

\section{Main Results}

In Remark \ref{gammaG5} we compute $\gamma_e^*(G_5)=2$ by observation, but as $k$ increases, the number of vertices increases exponentially and $\gamma_e^*$ becomes increasingly difficult to compute by brute force. Thus, our main results in this paper are upper and lower bounds for $\gamma_e^*(G_k)$. For all $k \geq 6$ we show that $U_{k-3}$ is a porous exponential dominating set for $G_k$, which proves the following:

\begin{theorem}
\label{UBgeq6}
For $k \geq 6$, $\gamma_e^*(G_k) \leq 3^{k-5}$.
\end{theorem}

We can improve upon this bound for $k \geq 11$ by constructing a porous exponential dominating set using all of the vertices of a smaller Apollonian network rather than just a generation. In particular, we dominate $G_k$ with $V(G_{k-7})$ and prove the following:

\begin{theorem}
\label{UBgeq10}
For $k \geq 10$, $\gamma_e^*(G_k) \leq \frac{3^{k-8}+5}{2}$.
\end{theorem}

To establish a lower bound, we apply a theorem from \cite{Dankelmann} that bounds $\gamma_e^*(G)$ from below in terms of $diam(G)$. In order to do this, we compute $diam(G_k)$ for all $k$. This establishes the following:

\begin{theorem}
\label{LB}
{For all $k \in \mathbb{N}$,  $\gamma_e^*(G_k) \geq $ $\left \lceil \frac{2k +5}{12} \right \rceil $}.
\end{theorem}

Before we can prove these theorems, we need some basic results about Apollonian networks.

\section{Apollonian Networks}

All of the vertices in $G_2$ are adjacent to each other, but for larger values of $k$, the adjacencies are more restrictive. Recall that $x$ is a \textit{neighbor} of $y$ in $G$ if $x$ is adjacent to $y$ in $G$, and the set of $y$'s neighbors in $G$ is the \textit{neighborhood} of $y$ in $G$, denoted $N_{G}(y)$.

\begin{lemma}
\label{U_k neighbors}
For all $k \geq 2$, and for every vertex $v$ in $U_k$,\\
(i) $v$ has no neighbor in $U_k$\\
(ii) $v$ has a neighbor in $U_{k-1}$\\
(iii) $v$ has exactly 3 distinct neighbors in $V(G_{k-1})$ and these vertices are also pairwise adjacent.\\
(iv) For all $r<k$ and for all $u \in U_r$, if $u$ is adjacent to $v$ then $|N_{G_k}(u) \cap N_{G_k}(v)|=2$.\\
(v) if $r<k$ and $v$ has more than one neighbor in $U_r$, then $r=1$
\end{lemma}

\begin{proof}
Parts (i), (ii), and (iii) follow directly from the construction of $G_k$ because when a new vertex $v$ is added to $U_k$, it is made adjacent to a vertex $u$ of $U_{k-1}$ and two of $u$'s neighbors in $V(G_{k-1})$, say $n_1$ and $n_2$. By part (iii), if one of $v$'s neighbors is $u$, then the other two are neighbors of both $u$ and $v$, and (iv) follows. We prove (v) by contradiction. Suppose that $1<r<k$ and two of $u$, $n_1$, and $n_2$ are in $U_r$. We know that $u \in U_{k-1}$, so if $k-1=r>1$ then $n_1, n_2 \not\in U_r$ by part (i). If $1<r<k-1$ then it must be that $n_1$ and $n_2$ are the two vertices in $U_r$. But by the construction of $G_{k-1}$, all three of $u$'s neighbors in $G_{k-2}$ (including $n_1$ and $n_2$) must be adjacent. This contradicts (i) for $n_1 \in U_r$ since $r>1$. 
\end{proof}

\begin{corollary}
\label{G_{k-3} neighbors}
For all $k \geq 4$ and for every vertex $ v \in U_k$, $v$ has at least one neighbor in $V(G_{k-3})$.
\end{corollary}

\begin{proof}
By Lemma \ref{U_k neighbors} part (iii), $v$ has exactly 3 distinct neighbors in $V(G_{k-1})$, and these vertices are also pairwise adjacent. Denote these vertices by $n_1$, $n_2$, and $n_3$, and suppose that $n_1 \in U_r$, $n_2 \in U_s$, and $n_3 \in U_t$, where $r \leq s \leq t \leq k-1$. Since $k \geq 4$, then $k-3 \geq 1$ and if $r>k-3$ then by pigeonhole principle, two of $r$, $s$, and $t$ must be equal which contradicts Lemma \ref{U_k neighbors} part (i). Therefore $r \leq k-3$ and $n_1 \in V(G_{k-3})$.
\end{proof}

Given $k \in \mathbb{N}$, $r \leq k$, and $v \in U_r$, define $P_{k}(v)=\{ \{x,y\} \mid x \in U_k \text{ and } v,x, \text{ and } y \text{ are pairwise adjacent}\}$. This is the set of pairs of vertices, at least one of which is from the $k$th generation, that form triangles with $v$ in $G_k$, the very same triangles that will anchor the $(k+1)$st generation of vertices. By the construction of $G_{k+1}$, there is a one-to-one corespondence between $P_{k}(v)$ and the $(k+1)$st generation neighbors of $v$. It follows that $|P_{k}(v)|=|N_{G_{k+1}}(V) \cap U_{k+1}|$, in other words the number of $(k+1)$st generation neighbors of $v$. The next lemma states that the number of such neighbors doubles with every generation.

\begin{lemma}
\label{P recursion}
For all $k \in \mathbb{N}$, for all $r \leq k$, and for all $v \in U_r$, $|P_{k+1}(v)|=2|P_{k}(v)|$.
\end{lemma}

\begin{proof}
By the construction of $G_{k+1}$, there is a one-to-one corespondence between $P_{k}(v)$ and the $(k+1)$st generation neighbors of $v$. It follows that the members of $P_{k+1}(v)$ are precisely the pairs $\{z,x\}$ and $\{z,y\}$ where $z \in U_{k+1} \cap N_{G_{k+1}}(v)$ and $\{x,y\} \in P_{k}(v)$.
\end{proof}

\begin{corollary}
\label{P size}
For all $k \in \mathbb{N}$, for all $r \leq k$, and for all $v \in U_r$.

\begin{equation*}
|P_{k} (v)|=
	\begin{cases}
		3(2^{k-r}) &\text{ when $ r > 1 $}\\
		2^{k-1} &\text{  when $ r=1 $}
	\end{cases}
\end{equation*}
\end{corollary}

\begin{proof}
We proceed by induction on $k$. If $k=1$ then $r=1$, then indeed for all $v \in U_1$, $|P_1(v)|=1=2^{1-1}$. If $k>1$ and $r=1$ then, by Lemma \ref{P recursion}, $|P_{k}(v)|=2|P_{k-1}(v)|=2(2^{k-2})=2^{k-1}$ by inductive hypothesis. If $k=2$ and $r=2$ then for the single vertex $v \in U_2$, $|P_2(v)|=3=3(2^{1-1})$. If $k>2$ and $r>1$ then, by Lemma \ref{P recursion}, $|P_{k}(v)|=2|P_{k-1}(v)|=2(3(2^{(k-1)-r}))=3(2^{k-r})$ by inductive hypothesis.
\end{proof}

\begin{corollary} 
\label{neighbors in U_k}
For all $k \geq 2$, and for all $v \in V(G_{k-1})$, $v$ has a neighbor in $U_k$.
\end{corollary}

\begin{proof} By the construction of $G_{k}$ there is a one-to-one corespondence between $P_{k-1}(v)$ and the $k$th generation neighbors of $v$. By Corollary \ref{P size} $|P_{k-1}(v)|$ is nonnegative, and therefore $v$ has a neighbor in $U_{k}$.
\end{proof}

\begin{corollary}
\label{neighbors in U_{k-1}}
For all $k \geq 2$, and for all $v \in V(G_k) \setminus U_{k-1}$, $v$ has a neighbor in $U_{k-1}$.
\end{corollary}

\begin{proof}
If $v \in U_k$ then the result follows immediately from Lemma \ref{U_k neighbors} part (ii). If $v \in U_r$, where $r \leq k-2$ then the result follows from Corollary \ref{neighbors in U_k}.
\end{proof}

\begin{lemma}
\label{PvsDegree}
For all $k \in \mathbb{N}$, for all $r \leq k$, and for all $v \in U_r$, 

\begin{equation*}
d_{G_k} (v) = \begin{cases}
|P_k(v)| &\text{ when $ r > 1 $}\\
|P_k(v)| + 1 &\text{  when $ r=1 $}
\end{cases}
\end{equation*}
\end{lemma}

\begin{proof}
By the construction of $G_{k+1}$, there is a one-to-one corespondence between $P_{k}(v)$ and the $(k+1)$st generation neighbors of $v$. It follows that for all $k \in \mathbb{N}$, for all $r \leq k$, and for all $v \in U_r$, \[d_{G_{k+1}}(v)=d_{G_{k}}(v)+|P_{k}(v)|.\] We now prove the lemma by induction on $k$. If $k=1$ then $r=1$ and $d_{G_{k}}(v)=2=1+1=|P_{k}(v)|+1$. If $k>1$ and $r=1$ then $d_{G_{k}}(v)=d_{G_{k-1}}(v)+|P_{k-1}(v)|=|P_{k-1}(v)|+1+|P_{k-1}(v)|=2|P_{k-1}(v)|+1=|P_{k}(v)|+1$ by inductive hypothesis and Lemma \ref{P recursion}. If $k=2$ and $r=2$ then for the single vertex $v \in U_2$, $d_{G_{k}}(v)=3=|P_{k}(v)|$. If $k>2$ and $r>1$ then $d_{G_{k}}(v)=d_{G_{k-1}}(v)+|P_{k-1}(v)|=|P_{k-1}(v)|+|P_{k-1}(v)|=2|P_{k-1}(v)|=|P_{k}(v)|$ by inductive hypothesis and Lemma \ref{P recursion}.
\end{proof}

\begin{corollary}
For all $k \in \mathbb{N}$, for all $r \leq k$, and for all $v \in U_r$, 

\begin{equation*}
d_{G_k} (v) = \begin{cases}
3(2^{k-r}) &\text{ when $ r > 1 $}\\
2^{k-1} + 1 &\text{  when $ r=1 $}
\end{cases}
\end{equation*}
\end{corollary}

\begin{proof}
This follows immediately from Corollary \ref{P size} and Lemma \ref{PvsDegree}.
\end{proof}

\section{Upper Bounds for $\gamma_e^*$}
 
In \cite{Dankelmann} the nonporous exponential dominating number of G, denoted $\gamma_e(G)$, is defined and the following theorem is proved:

\begin{theorem} (Dankelmann, et al) If G is a connected graph of order n, then $\gamma_e(G) \leq \frac{2}{5} (n+2)$.
\end{theorem}

This theorem, together with Remark \ref{Vertices} and the fact that $\gamma_e^*(G) \leq \gamma_e(G)$, immediately establishes the following corollary:

\begin{corollary}
For all $k \in \mathbb{N}$, $\gamma_e^*(G_k) \leq \frac{3^{k-1}+9}{5}$.
\end{corollary}

The recursive nature of our construction of $G_k$ makes it clear that for, $k>1$, $G_k$ can be conceived as a union of three copies of $G_{k-1}$. More precisely, if we consider the three triangles in $G_2$ that include the vertex in $U_2$, each could be the first generation of a copy of $G_{k-1}$. Together, these three copies of $G_{k-1}$ comprise a copy of $G_k$. This perspective is also discussed in \cite{Zhang2014}. The following lemma follows immediately from this construction.

\begin{lemma}
\label{3recursion}
For all $k \in \mathbb{N}$, $\gamma_e^*(G_{k+1}) \leq 3 \gamma_e^*(G_k)$.
\end{lemma}

\begin{corollary}
For $k \geq 5$, $\gamma_e^*(G_k) \leq 2(3^{k-5})$.
\end{corollary}

\begin{proof} By induction on $k$. If $k=5$ then the result follows immediately from Remark \ref{gammaG5}. If $k>5$ then by Lemma \ref{3recursion}, $\gamma_e^*(G_k) \leq 3 \gamma_e^*(G_{k-1})=3(2(3^{(k-1)-5}))=2(3^{k-5})$ by inductive hypothesis.
\end{proof}

We now establish a better upper bound by proving Theorem \ref{UBgeq6}:

\begin{proof}
Suppose $k \geq 6$. Let $S=U_{k-3}$ and compute $w_S^*(v)$ for all $v \in V(G_k)$.

\medskip

Case 1: Suppose $v\in V(G_{k-4})$. By Corollary \ref{neighbors in U_k}, $v$ has a neighbor in $S$ and $w^*_S(v) \geq 1$.

\medskip

Case 2: Suppose $v \in U_{k-3}$. Then $v \in S$ and $w^*_S(v) \geq 2$.

\medskip

Case 3: Suppose $v \in U_{k-2}$. By Corollary \ref{neighbors in U_{k-1}}, $v$ has a neighbor in $S$ and $w^*_S(v) \geq 1$.

\medskip

Case 4: Suppose $v \in U_{k-1}$ or $v \in U_{k}$. By Lemma $\ref{U_k neighbors}$, $v$ has three distinct neighbors in $V(G_{k-1})$. If $v$ has a neighbor in $S$ then $w^*_S(v) \geq 1$. Otherwise, at least one of $v$'s neighbors is in $V(G_{k-4})$. Let $n$ be this vertex. By Corollary \ref{P size}, $n$ has more than one neighbor in $S$. Therefore, $v$ is within distance 2 of at least two distinct vertices of  $S$, and $w^*_S(v) \geq 1$.

\medskip

We have shown that $S$ is a porous exponential dominating set for $G_k$. By Remark \ref{Vertices}, $|S|=3^{k-5}$, and therefore $\gamma_e^*(G_k) \leq 3^{k-5}$.
\end{proof}

We proved Theorem \ref{UBgeq6} by using a particular generation as a porous exponential dominating set. For $k \geq 10$, we can improve this upper bound by using the entire vertex set of a smaller Apollonian network as a dominating set. This is the strategy we employ in the proof of Theorem \ref{UBgeq10}:

\begin{proof} Suppose $k \geq 10$. Let $S=V(G_{k-7})$ and compute $w_S^*(v)$ for all $v \in V(G_k)$.

\medskip

Case 1: Suppose $v \in U_j$, $j \leq k-4$. Then by Corollary \ref{G_{k-3} neighbors}, either $v \in S$ or $v$ has a neighbor in $S$. In both cases, $w^*_S(v) \geq 1$.

\medskip

Case 2: Suppose $v \in U_j$, $k-3 \leq j \leq k-2$. If $v$ has a neighbor in $S$, then $w^*_S(v) \geq 1$. Otherwise, by Corollary \ref{G_{k-3} neighbors}, $v$ has a neighbor $n$ in either $U_{k-5}$ or $U_{k-6}$. By Lemma \ref{U_k neighbors}, $n$ has at least two neighbors in $S$. Therefore, $v$ is within distance 2 of at least two distinct vertices of  $S$, and $w^*_S(v) \geq 1$.
 
\medskip

Case 3: Suppose $v \in U_{k-1}$. If $v$ has a neighbor in $S$, then $w^*_S(v) \geq 1$. Otherwise, by Corollary \ref{G_{k-3} neighbors}, $v$ has a neighbor $n$ in $U_{n-4}$, $U_{n-5}$, or $U_{n-6}$. By Corollary \ref{G_{k-3} neighbors}, $n$ has a neighbor $w \in S$. If $w \in U_1$ then $w$ has two neighbors $x,y \in U_1$. Note that $w,x,y \in S$, and that $v$ is within distance 2 of $w$ and within distance 3 of each of $x$ and $y$. Therefore $w^*_S(v) \geq \frac{1}{2} + \frac{1}{4} + \frac{1}{4} \geq 1$. Otherwise $w \in U_j$, $j \geq 2$, and by Lemma \ref{U_k neighbors} $w$ has three distinct neighbors $x_1,x_2,x_3 \in V(G_{j-1})$. Note that $w,x_1,x_2,x_3 \in S$, and that $v$ is within distance 2 of $w$ and within distance 3 of each of $x_1$, $x_2$, and $x_3$. Therefore $w^*_S(v) \geq \frac{1}{2} + \frac{1}{4} + \frac{1}{4} + \frac{1}{4} \geq 1$.

\medskip

Case 4: Suppose $v \in U_{k}$. If $v$ has a neighbor in $S$, then $w^*_S(v) \geq 1$. Otherwise, by Corollary \ref{G_{k-3} neighbors}, $v$ has a neighbor $n$ in $U_{n-3}$, $U_{n-4}$, $U_{n-5}$, or $U_{n-6}$. By Corollary \ref{G_{k-3} neighbors}, $n$ has a neighbor $w$ such that $w \in U_{k-6}$ or $w \in S$. If $w \in S$ then proceed as in Case 3. If $w \in U_{k-6}$ then by Lemma \ref{U_k neighbors} $w$ has three distinct neighbors $x_1,x_2,x_3 \in S$. Let $x_3$ be the neighbor with smallest generation. Since $k \geq 10$, by Corollary \ref{P size} and Lemma \ref{U_k neighbors} part (v), $x_3$ has at least 3 neighbors $y_1,y_2,y_3 \in U_{k-7}$ distinct from $x_1$ and $x_2$. Note that $x_1,x_2,x_3,y_1,y_2,y_3 \in S$. Also note that $v$ is within distance 3 of each of $x_1$, $x_2$, and $x_3$, and within distance 4 of each of $y_1$, $y_2$, and $y_3$. Therefore $w^*_S(v) \geq \frac{1}{4} + \frac{1}{4} + \frac{1}{4} + \frac{1}{8} + \frac{1}{8} + \frac{1}{8}\geq 1$.
 
\medskip

We have shown that $S$ is a porous exponential dominating set for $G_k$. By Remark \ref{Vertices}, $|S|=\frac{3^{k-8}+5}{2}$, and therefore $\gamma_e^*(G_k) \leq \frac{3^{k-8}+5}{2}$.
\end{proof}

\section{Lower Bound for $\gamma_e^*$}

Recall that for a connected graph $G$, the diameter of $G$, denoted $diam(G)$, is the largest possible distance between a pair of vertices in $G$. In \cite{Dankelmann} 
 the nonporous exponential domination number of $G$, denoted $\gamma_e(G)$, is defined and the following theorem is proven: 

\begin{theorem} (Dankelmann, et al)
If $G$ is a connected graph, then 
$\gamma_e(G) \geq \left \lceil \frac{diam(G)+2}{4}\right \rceil$.
\end{theorem}

In fact, the proof of this result in \cite{Dankelmann} is sufficient to establish the following lemma: 

\begin{lemma}
\label{DiamBound}
If $G$ is a connected graph, then 
$\gamma_e^*(G) \geq \left \lceil \frac{diam(G)+2}{4}\right \rceil$.
\end{lemma}

We now compute $diam(G_k)$ for every Apollonian network $G_k$.

\begin{lemma}
\label{Diam UB Recursion}
For all $k \in \mathbb{N}$, $diam(G_{k+3}) \leq diam(G_k)+2$.
\end{lemma}

\begin{proof}
Suppose $x,y \in V(G_{k+3})$ and $d_{G_{k+3}}(x,y)=diam(G_{k+3})$. By Lemma \ref{U_k neighbors} and Corollary \ref{G_{k-3} neighbors}, we know that $x$ and $y$ have neighbors $u$ and $v$, respectively, in $V(G_k)$. It follows that \[diam(G_{k+3})=d_{G_{k+3}}(x,y) \leq d_{G_{k}}(u,v)+2 \leq diam(G_{k})+2.\]
\end{proof}

\begin{corollary}
\label{Diam UB}
For all $k \in \mathbb{N}$, 
$diam(G_k) \leq \left \lfloor{\frac{2k+1}{3}}\right \rfloor$.
\end{corollary}

\begin{proof}
We proceed by induction on $k$ and show that $diam(G_k) \leq \frac{2k+1}{3}$. For $k=1,2,3$, it is easy to verify that $diam(G_k)=1,1,2$, respectively, and establish the desired result. For $k>3$, by Lemma \ref{Diam UB Recursion}, $diam(G_k) \leq diam(G_{k-3})+2 \leq \frac{2(k-3)+1}{3}+2=\frac{2k+1}{3}$, by inductive hypothesis. Since $diam(G_k)$ is an integer, the result follows.
\end{proof}

\begin{lemma}
\label{Uk diam path}
For all $k \in \mathbb{N}$ there exists $x,y \in U_k$ such that $d_{G_{k}}(x,y)=diam(G_k)$.
\end{lemma}

\begin{proof}
First, observe that the statement is true for $k=1$, so we may assume $k \geq 2$. Let $u,v \in V(G_{k})$ such that $d_{G_{k}}(u,v)=diam(G_k)$. If $u \in U_k$ then let $x=u$. Otherwise, by Corollary \ref{neighbors in U_k}, there exists $x \in U_k$ such that $x$ is adjacent to $u$. If $v \in U_k$ then let $y=v$. Otherwise, by Corollary \ref{neighbors in U_k}, there exists $y \in U_k$ such that $y$ is adjacent to $v$. Let $P$ be a shortest path joining $x$ and $y$. Let $w_1$ be the vertex adjacent to $x$ in $P$, and $w_2$ be the vertex adjacent to $y$ in $P$. By Lemma \ref{U_k neighbors} part (iii), $u$ is adjacent to $w_1$ and $v$ is adjacent to $w_2$. Define $Q$ to be the path formed by replacing $x$ and $y$ in $P$ with $u$ and $v$. Then the length of $Q$ is the same as the length of $P$. Since $d_{G_{k}}(u,v)=diam(G_k)$, this shows that the length of $P$ is at least $diam(G_k)$. Since $P$ is a shortest path joining $x$ and $y$, $d_{G_{k}}(x,y)=diam(G_k)$.
\end{proof}

\begin{lemma}
\label{Diam LB Recursion}
For all $k \in \mathbb{N}$, $diam(G_{k+3}) \geq diam(G_k)+2$.
\end{lemma}

\begin{proof}
The result is easily seen to be true for $k=1$, so we may assume that $k \geq 2$. (See Figure \ref{fig:G4} and Table \ref{GkSmall}.) By Lemma \ref{Uk diam path}, let $u,v \in V(G_{k})$ such that $d_{G_{k}}(u,v)=diam(G_{k})$. By Lemma \ref{U_k neighbors}, any path joining $u$ and $v$ must include vertices from $V(G_{k-1})$. By Corollary \ref{neighbors in U_k}, $u$ has a neighbor $u_1$ in $U_{k+1}$. By the construction of $G_{k+2}$, $u$ and $u_1$ have a common neighbor $u_2$ in $U_{k+2}$. By the construction of $G_{k+3}$, $u$, $u_1$, and $u_2$ have a common neighbor $x$ in $U_{k+3}$. By Lemma \ref{U_k neighbors}, $u$, $u_1$, and $u_2$ are the only neighbors of $x$ in $G_{k+3}$. Therefore, $u$ has a neighbor $x \in U_{k+3}$ such that $N_{G_{k+3}}(x) \cap V(G_{k-1}) = \emptyset$. An analogous argument shows that $v$ has a neighbor $y \in U_{k+3}$ such that $N_{G_{k+3}}(y) \cap V(G_{k-1}) = \emptyset$. Note that any path joining $x$ and $y$ must include vertices from $V(G_{k-1})$ because otherwise we could construct a path joining $u$ and $v$ without such vertices, which contradicts our earlier claim to the contrary.

Let $P$ be a shortest path $x,w_1,w_2,\ldots,w_m,y$ joining $x$ and $y$. Choose $i$ as small as possible and $j$ as large as possible such that $w_i,w_j \in V(G_{k-1}) \cap V(P)$. Since the only neighbors of $x$ are $u$, $u_1$, and $u_2$ then $w_i$ is adjacent to at least one of these. By the construction of $G_{k+1}$ and $G_{k+2}$, $u$ is adjacent to all of the neighbors of $u_1$ and $u_2$ in $V(G_{k-1})$, and therefore $u$ is adjacent to $w_i$. Analogously, $v$ is adjacent to $w_j$. Let $Q$ be the path $u,w_i,w_{i+1},\ldots,w_{j-1},w_j,v$ joining $u$ and $v$. Since $N_{G_{k+3}}(x) \cap V(G_{k-1}) = \emptyset$ and $N_{G_{k+3}}(y) \cap V(G_{k-1}) = \emptyset$, the length of $P$ is at least 2 more than the length of $Q$. It follows that the length of $P$ is at least $diam(G_k)+2$. Since $P$ is a shortest length path joining $x$ and $y$, $diam(G_{k+3}) \geq diam(G_k)+2$.
\end{proof}

Together, Lemma \ref{Diam UB Recursion} and Lemma \ref{Diam LB Recursion} imply the following result which was stated in \cite{Zhang2006} with greater generality but without a complete proof.

\begin{corollary}
\label{Diam Exact Recursion}
For all $k \in \mathbb{N}$, $diam(G_{k+3}) = diam(G_k)+2$.
\end{corollary}

\begin{corollary}
\label{Diam LB}
For all $k \in \mathbb{N}$, 
$diam(G_k) \geq \left \lceil{\frac{2k-1}{3}}\right \rceil$.
\end{corollary}

\begin{proof}
We proceed by induction on $k$ and show that $diam(G_k) \geq \frac{2k-1}{3}$. For $k=1,2,3$, it is easy to verify that $diam(G_k)=1,1,2$, respectively, and establish the desired result. For $k>3$, by Lemma \ref{Diam LB Recursion}, $diam(G_k) \geq diam(G_{k-3})+2 \geq \frac{2(k-3)-1}{3}+2=\frac{2k-1}{3}$, by inductive hypothesis. Since $diam(G_k)$ is an integer, the result follows.
\end{proof}

\begin{corollary}
\label{Diam Exact}
For all $k \in \mathbb{N}$, 
$diam(G_k) = \left \lceil{\frac{2k-1}{3}}\right \rceil$.
\end{corollary}

\begin{proof}
This result follows easily from Corollary \ref{Diam UB}, Corollary \ref{Diam LB}, and the fact that
$ \left \lfloor{\frac{2k+1}{3}}\right \rfloor = \left \lceil{\frac{2k-1}{3}}\right \rceil $, which the reader can easily check by cases $k \equiv 0,1,2$ (mod 3).
\end{proof}

We can now prove Theorem \ref{LB}:

\begin{proof}
By Corollary \ref{Diam Exact}, $diam(G_k) = \left \lceil{\frac{2k-1}{3}}\right \rceil  \geq {\frac{2k-1}{3}}$, and therefore  $\frac{diam(G_k)+2}{4} \geq \frac{2k +5}{12} $. By Lemma \ref{DiamBound},
$\gamma_e^*(G_k) \geq \left \lceil \frac{diam(G_k)+2}{4}\right \rceil \geq \left \lceil \frac{2k +5}{12} \right \rceil$.
\end{proof}

\section{Acknowledgements}

The authors would like to thank CURM, the Center for Undergraduate Research, for facilitating a wonderful undergraduate research experience and funding our research through NSF grant DMS-1148695. We would also like to thank Concord University for their encouragement and financial support.

%%%%%%%%%%%%%%%
% Bibiography %
%%%%%%%%%%%%%%%

\footnotesize{

} 


\begin{thebibliography}{99}

\bibitem[1]{Anderson} {\sc M. Anderson, R. Brigham, J. Carrington, R. Vitray, J. Yellen}, On Exponential Domination of $C_m \times C_n$, {\em J. Graphs. Combin.} 6 No. 3 (2009) 341-351.

\bibitem[2]{Andrade} {\sc J. Andrade, H. Herrmann, R. Andrade, L. Silva}, Simultaneously Scale-Free, Small World, Euclidean, Space Filling, and with Matching Graphs, {\em Phys. Rev. Lett.} 94(1) (2005) 18702.

\bibitem[3]{Dankelmann} {\sc P. Dankelmann, D. Day, D. Erwin, S. Mukwembi, H. Swart}, Domination with Exponential Decay, {\em Discrete Math} 309(19) (2009) 5877–5883.

\bibitem[4]{Doye} {\sc J.P.K. Doye and C. Massen}, Characterizing the network topology of the energy landscapes of atomic clusters {\em J. Chem. Phys.} 122 (2005) 84105.

\bibitem[5]{HaynesDom} {\sc T.W. Haynes, S.T. Hedetniemi, P.J. Slater}, {\em Domination in Graphs: Advanced Topics}, Marcel Dekker, New York, 1998.

\bibitem[6]{HaynesFund} {\sc T.W. Haynes, S.T. Hedetniemi, P.J. Slater}, {\em Fundamentals of Domination in Graphs}, Marcel Dekker, New York, 1998.

\bibitem[7]{West} {\sc D. B. West}, {\em Introduction to Graph Theory, 2$^{nd}$ ed.}, Prentice Hall, 2000.

\bibitem[8]{Zhang2006} {\sc Z. Zhang, F. Comellas, G. Fertin, L. Rong}, High Dimensional Apollonian Networks, {\em J. Phys. A: Math. Gen.} 39 (2006) 1811.

\bibitem[9]{Zhang2014} {\sc Z. Zhang, B. Wu, and F. Comellas}, The Number of Spanning Trees in Apollonian Networks, {\em Discrete Applied Mathematics}. (to appear)

\end{thebibliography}
\end{document}